\documentclass[11pt,a4paper, twocolumn]{article}

\usepackage[final]{pdfpages}
\usepackage[T1]{fontenc}
\usepackage[utf8]{inputenc}
\usepackage{mathptmx}
\usepackage{amsmath,amssymb,amsthm}
\usepackage{amsfonts,latexsym}
\usepackage{amscd,array}
\usepackage{float,mathrsfs,bm,multirow,graphicx,verbatim}
\usepackage{authblk}

\theoremstyle{definition}
\newtheorem{definition}{Definition}[section]
\newtheorem{assumption}{Assumption}[section]
\theoremstyle{remark}
\newtheorem*{remark}{Remark}
\newtheorem{proposition}{Proposition}[section]

\newcommand{\be}{\begin{equation}}
\newcommand{\ee}{\end{equation}}

\title{Merits of the Incremental Method for modeling Piecewise Linear functions}

\author[1]{Mutaz Tuffaha \thanks{{\tt\small mutaz.tuffaha@ntnu.no}}}
\author[1]{Jan Tommy Gravdahl \thanks{{\tt\small jan.tommy.gravdahl@ntnu.no}}}
\affil[1]{Department of Engineering Cybernetics, Norwegian University of Science and Technology, NO-7491 Trondheim-Norway}

\date{}
\begin{document}
\maketitle
\begin{abstract}
Several techniques were proposed to model the \textit{Piecewise linear} (PWL) functions, including convex combination, incremental and multiple choice methods. Although the incremental method was proved to be very efficient, the attention of the authors in this field was drawn to the convex combination method, especially for discontinuous PWL functions. In this work, we modify the incremental method to make it suitable for discontinuous functions. The numerical results indicate that the modified incremental method could have considerable reduction in computational time, mainly due to the reduction in the number of the required variables. Further, we propose a tighter formulation for optimization problems over separable univariate PWL functions with binary indicators by using the incremental method.
\end{abstract}
\section{INTRODUCTION}
\textit{Piecewise Linear} (PWL) approximation of the objective function has been one of the most pragmatic techniques used in the field of the optimization theory and \textit{operation research} (OR). In general, any univariate function defined on the interval $[\underline{x},\overline{x}]$ can be approximated, to good accuracy \cite{Farias2}, by its PWL approximation. If the interval $[\underline{x},\overline{x}]$ is partitioned into $K$ segments with $K+1$ breaking points $\underline{x}=a_0<a_1<\dots<a_K=\overline{x}$, then the PWL approximation $\tilde{f}(x):[a_0,a_K]\longrightarrow\mathbb{R}$ of the function $f(x):[a_0,a_K]\longrightarrow\mathbb{R}$ can be described by:  
\be
\label{eq:PWLf}
\tilde{f}(x)=\left\{\begin{array}{lr}
m_1x+d_1,&a_0\leq x\leq a_1\\
m_2x+d_2,&a_1\leq x\leq a_2\\
\vdots&\\
m_Kx+d_K,&a_{K-1}\leq x\leq a_K
\end{array}\right\},
\ee 
where $m_1,...,m_K$, and $d_1,...,d_K$ are the slopes and cost-intercepts of the line segments, respectively. The accuracy of this approximation can be controlled by the number and location of the breaking points used. For more information on the accuracy issue, the readers are referred to \cite{Frenzen}, and \cite{Sypros}. 
\newline
The optimization over PWL separable functions has been studied extensively in literature.  The authors in \cite{Croxton} and \cite{Sridhar}, e.g. provided good surveys on the \textit{Mixed-Integer Programming} (MIP) methods used to model the PWL functions in optimization problem. While proposing a technique for optimization over non-separable PWL functions, the authors in \cite{Vielma1} wrote an excellent literature review, as well. The MIP models for PWL functions in optimization problems can be classified under three main methods: \textit{incremental}, \textit{convex combination}, and \textit{multiple choice}.
\subsection{Incremental Method}
According to \cite{Padberg} and \cite{Croxton}, the incremental method is attributed to the pioneering work in \cite{Markowitz} and \cite{Dantzig}, and it is usually referred to in literature as the \textit{Delta} method. In this method, each variable $x$ is expressed as the sum \cite{Padberg}:
\be
\label{eq:xIM}
x=a_0+\sum_{k=1}^K{y_k},
\ee  
where $y_k,~\forall k\in\{1,..,K\}$ are continuous variables that satisfy the following two conditions \cite{Padberg}:
\begin{enumerate}
\item $0\leq y_k \leq a_k-a_{k-1}$,
\item if $y_k<a_k-a_{k-1}$, then $y_{k+1}=0$.
\end{enumerate}
The second condition above is called a \textit{dichotomy}, which indicates that it is a bi-directional decision, either $y_k=a_k-a_{k-1}$ or $y_{k+1}=0$. The proposed technique to enforce the second condition is to use some binary variables $\beta_k\in\{0,1\},~\forall k\in\{1,...,K-1\}$, such that \cite{Padberg}:
\be
\label{con:Deltay_1}
y_1\leq a_1-a_0\ee
and
\begin{align}
y_K\geq& 0\nonumber\\
y_k\geq&(a_k-a_{k-1})\beta_k,\hspace{0.5cm} \forall k\in\{1,..K-1\}\nonumber\\
\label{con:Delta}
y_k\leq&(a_k-a_{k-1})\beta_{k-1},\hspace{0.5cm}\forall k\in\{2,..K\}. 
\end{align}
Thus, if $\beta_{k-1}=0$ then the only feasible solution will be if $\beta_k=0$. Accordingly, if the PWL objective functions is continuous, it can be expressed with the continuous variables $y_k$'s
as (see e.g. \cite{Padberg}):
\be
\label{eq:ObjDelta}
\tilde{f}(x)=f(a_0)+\sum_{k=1}^K{m_ky_k}.
\ee
\subsection{Convex Combination}
This method is based on the fact that if the variable $x$ lies in the interval $[a_k,a_{k+1}]$, then $x$ can be written as a convex combination of the two consecutive points $a_k$ and $a_{k+1}$, since the domain of $x$ is convex. Thus (see e.g. \cite{Sridhar}):
\be
\label{eq:xCC}
x=a_0\lambda_0+a_1\lambda_1+...+a_K\lambda_K,
\ee  
where $\lambda_k,~\forall k\in\{0,..,K\}$ are continuous variables or weights that satisfy the following two constraints (see e.g. \cite{Padberg}):
\begin{enumerate}
\item The sum of the weights should be equal to one, i.e.:
\be
\label{con:CCwieghts}
\sum_{k=0}^K{\lambda_k}=1.
\ee
\item At most two consecutive elements of the set $\{\lambda_k\in\mathbb{R}_+:~\forall k\in\{0,...,K\}\}$ are non-zero.
\end{enumerate}
The authors in \cite{Tomlin} defined the sets like the one above as a \textit{Special Order Set of type 2} (SOS2). Several algorithms were proposed in the literature to satisfy the SOS2 condition. At the beginning, the authors in \cite{Tomlin} proposed a modification in the \textit{Branch and Bound} (BB) algorithm to capture this condition. In the MIP model by the convex combination method, the common practice is to introduce some binary variables to enforce the second condition above. For example, one can use the formulation given in \cite{Padberg}, \cite{Farias2} and \cite{JonLee} in which, the binary variables $\beta_k\in\{0,1\},~\forall k\in\{1,...,K\}$ are used with the following constraints:
\begin{align}
\label{con:Lambda}
\sum_{k=1}^K{\beta_k=1}&\nonumber\\
\lambda_0\leq \beta_1,& \hspace{0.8cm}\lambda_K\leq\beta_K\nonumber\\
\lambda_k\leq \beta_k+\beta_{k+1},&\hspace{1.2cm}\forall k\in\{1,...,K-1\}.
\end{align}
Then, if the PWL objective functions is continuous, it can be expressed with the weights $\lambda_k$ as:
\be
\label{eq:ObjLambda}
\tilde{f}(x)=\sum_{k=0}^K{f(a_k)\lambda_k}.
\ee
The above method is sometimes called the \textit{Lambda} method.  
\begin{remark}
\label{rem:NumVar}
For a PWL function defined on one variable $x$, the incremental method requires $K$ continuous variables and $K-1$ binary variables. On the other hand, the convex combination method requires $K+1$ continuous variables, and $K$ binary variables.
\end{remark}
The \textit{multiple choice} method is described in \cite{Croxton} and \cite{Vielma1} among others, but it is beyond the scope of this work.
\newline
In his pioneer work, the author in \cite{Padberg} proved that the incremental method gives a \textit{locally} ideal model, while the convex combination method with the constraints given in \eqref{con:Lambda} does not. To elucidate, the extreme points of the polytope of the \textit{Linear Programming} (LP) relaxation of the model described by the constraints in \eqref{con:Deltay_1}-\eqref{con:Delta} are integers in variables $\beta_k$. In any MIP model, one looks for such features because the optimal solutions can be found from LP relaxation without the need for BB technique. The term "\textit{locally}" refers here to the fact that this feature is guaranteed only when the optimization problem is solved over one variable $x$, while in the case of several variables or constraints on the original variable $x$ the ideality is not guaranteed. In contrast, the author in \cite{Padberg} proved that the extreme points of the polytope of the LP relaxation of the model described by the constraints in \eqref{con:Lambda} may contain non-integer values in $\beta$, and hence the model is not ideal. Besides, he described an alternative model to implement the SOS2 condition that keeps the ideality property. Later, the authors in \cite{Farias2} proved that implementing the SOS2 condition by the strategy proposed in \cite{Tomlin} makes the model locally ideal.
\begin{remark}
The authors on this topic distinguish between the convex combination model described by \eqref{eq:xCC}, \eqref{con:CCwieghts}, \eqref{con:Lambda}, and \eqref{eq:ObjLambda} and the SOS2 model described by \eqref{eq:xCC}, \eqref{con:CCwieghts}, \eqref{eq:ObjLambda}, and the SOS2 condition of the set $\{\lambda_k~\forall~k\}$. The difference is that in the SOS2 model, the branching technique in the BB algorithm is modified to achieve this constraint, while in the convex combination model a MIP model is used by introducing some binary indicators, as shown in \eqref{con:Lambda}. As mentioned before, the convex combination model is not locally ideal unless the model is modified as proposed in \cite{Padberg}, while the SOS2 model is. 
\end{remark}
Many researchers have been discussing the techniques used to model PWL functions, and their challenges. Of the most interesting challenges in PWL models is the discontinuity of the function, and the optimization problem over separable univariate PWL functions with binary indicators. The methods described in \cite{Padberg} are well-suited for continuous functions, but some modification must be made when the function is discontinuous. Further, those methods are designed for optimizing over one variable $x$, but what if the problem involves many variables with binary indicators? For example, consider the following optimization problem described by:
\be
\label{eq:BinInd}
\min_{\mathbf{x}}\hspace{0.5 cm}F(\mathbf{x})~\text{s.t.}~\mathbf{x}\in\Omega \cup \{\mathbf{0}\},
\ee
where $\mathbf{x}=(x_1,\dots,x_N)^T$, and $\Omega$ is some feasibility region. Actually, many suggestions have been proposed in the literature to cope with those challenges, as will be discussed later. However, most of the researchers in this field focus on the convex combination method because it is the commonly used one. Spurred by the fact that the incremental method gives a locally ideal models for PWL functions, and the fact that it requires less number of variables, as indicated in Remark~\ref{rem:NumVar}, we try in this work to tackle the aforementioned challenges by the incremental method. 
\newline
The paper is organized as follows. In the next section, we propose a modification of the incremental method to make it suit the discontinuous functions with some numerical examples. The third section discusses the optimization problems of univariate separable functions with binary indicators by using the incremental method. In the fourth section, we sum up our conclusions.  
\section{Discontinuity}
It is well known that a discontinuous function may not have minima, in particular when it is not \textit{lower semi-continuous}.
\begin{definition}[Lower Semi-continuity, \cite{Farias3}]
\label{def:LowerSemi}
A univariate function $h(x):[0,u]\longrightarrow \mathbb{R}$ where $u>0$ is called lower semi-continuous if $h(x)\leq \lim_{x^{\prime}\longrightarrow x}\hspace{0.1cm}\inf\hspace{0.2cm} h(x^{\prime})$, where $\{x^{\prime}\}$ is a sequence in the domain of $x$
\end{definition}
Thus, if the function is not lower semi-continuous a minimum does not exist. Similarly, functions that are not upper semi-continuous, do not have maxima. The discontinuity of the PWL function was first handled in \cite{Hanif} by assigning two weights around each discontinuity from the left and right, and then the value of the function can be approximated by the convex combination of those weights. In spite of the cleverness of the proposed method in \cite{Hanif}, it adds to the number of variables required to model the PWL function. Later, the authors in \cite{Farias3} proposed a new type of ordered sets instead of SOS2 and they called it \textit{Special Order Set of type D} (SOSD), that can be implemented by modifying the branching technique. The method proposed in \cite{Farias3} can be used for continuous or discontinuous PWL functions, whether they are lower semi-continuous or not. However, the number of variables required for this model is also high. Besides, the proposed method requires changing the branching technique in the BB algorithm, which could be more complicated.
\newline
Before presenting our main results, we need the following definition.
\begin{definition}
A function $f(x):D\subset\mathbb{R} \longrightarrow\mathbb{R}$ is said to be right-continuous at $x_0\in D$, if $\lim_{x\longrightarrow x_0^+}f(x)=f(x_0)$. Moreover, the function $f(x)$ is said to be right-continuous if it is right-continuous $\forall x\in D$.
\newline
Similarly, it is said to be left-continuous if it is left-continuous $\forall x\in D$.
\end{definition}
Now, we show that any discontinuous PWL function can be modelled by the incremental method under the following assumption.
\begin{assumption}
The discontinuous PWL function $\tilde{f}(x)$ is either right-continuous or left-continuous.
\end{assumption}  
\subsection{Right-Continuous}
Let $\tilde{f}(x):[a_0,a_K]\longrightarrow\mathbb{R}$ be a PWL function described by:  
\be
\label{eq:PWLfConRight}
\tilde{f}(x)=\left\{\begin{array}{lr}
m_1x+d_1,&a_0\leq x<a_1\\
m_2x+d_2,&a_1\leq x<a_2\\
\vdots&\\
m_Kx+d_K,&a_{K-1}\leq x\leq a_K
\end{array}\right\}.
\ee 
Unlike the function in \eqref{eq:PWLf}, the function above is only right-continuous at all breaking points but not from the left. Then, we propose to use the same incremental model given in \eqref{eq:xIM}, \eqref{con:Deltay_1} and \eqref{con:Delta}, with a simple modification of the function representation in \eqref{eq:ObjDelta} into:
\be
\label{eq:ObjDeltaContRight}
\tilde{f}(x)=\tilde{f}(a_0)+\sum_{k=1}^K{\left(m_ky_k+\Delta_k\beta_k\right)},
\ee
where $\Delta_k$ represents the jump at the point $a_k$, and it is given by $\tilde{f}(a_k)-(m_ka_k+d_k)$. To elaborate, the binary variable $\beta_k$ in the incremental method is supposed to be one only when the continuous variable $y_k=a_k-a_{k-1}$, and zero otherwise. Thus, when $y_k<a_k-a_{k-1}$ the jump will not be counted because $\beta_k=0$. Needless to say, this model can still be used for continuous functions because in this case $\Delta_k=0$, i.e. there is no jump. 
\begin{remark}
The model described by \eqref{eq:xIM}, \eqref{con:Deltay_1}, \eqref{con:Delta}, and \eqref{eq:ObjDeltaContRight} for discontinuous PWL functions that are right-continuous is still locally ideal.
Obviously, modifying \eqref{eq:ObjDelta} into \eqref{eq:ObjDeltaContRight} does not affect the local ideality of the incremental model because the polytope of the LP relaxation of the model described by the constraints in \eqref{con:Deltay_1} and \eqref{con:Delta} is not changed, and hence the extreme points are still integral in $\beta_k~\forall~k$ as proved in \cite{Padberg}.
\end{remark}
\subsubsection{Numerical Example}
Let us consider the PWL function $\tilde{f}(x)$ defined as follows:
\be
\label{eq:ExampleFun1}
\tilde{f}(x)=\left\{\begin{array}{lr}
-5x+7.5,&0\leq x<1\\
-5x+15,&1\leq x<2\\
-2.5x+12.5,&2\leq x\leq 3
\end{array}\right\}.
\ee 
In order to build the incremental model of this PWL function by using \eqref{eq:xIM}, \eqref{con:Deltay_1} and \eqref{con:Delta}, we need three continuous variables $y_k~\in~\mathbb{R}, ~\forall~k\in\{1,2,3\}$, and two binary variables $\beta_k~\in~\{0,1\}, ~\forall~k\in\{1,2\}$. Then, the model can be described by:
\be
\label{eq:xIMEx1}
x=y_1+y_2+y_3
\ee
with the constraints:
\begin{align}
\label{con:Deltay_1Ex1}
\beta_1\leq y_1\leq 1\nonumber\\
\beta_2\leq y_2\leq \beta_1\nonumber\\
0\leq y_3\leq \beta_2.
\end{align}
The objective function $\tilde{f}(x)$ in \eqref{eq:ExampleFun1} can be expressed by the proposed model in \eqref{eq:ObjDeltaContRight} as:
\be
\label{eq:ObjIMEx1}
\tilde{f}(x)=7.5-5y_1-5y_2-2.5y_3+7.5\beta_1+2.5\beta_2.
\ee
For comparison purpose, we use the convex combination model described in \cite{Hanif} which was proved to be locally ideal. According to the model presented in \cite{Hanif}, one needs six continuous variables $\lambda_k~\in~\mathbb{R}, ~\forall~k\in\{0,...,5\}$ and three binary variables $\beta_k~\in~\{0,1\}, ~\forall~k\in\{1,2,3\}$. Then, the model is described as \cite{Hanif}: 
\be
\label{eq:xCCEx1}
x=\lambda_1+\lambda_2+2(\lambda_3+\lambda_4)+3\lambda_5
\ee
with the constraints:
\begin{align}
\label{con:LambdaEx1}
\lambda_0+\lambda_1=y_1\nonumber\\
\lambda_2+\lambda_3=y_2\nonumber\\
\lambda_4+\lambda_5=y_3\nonumber\\
y_1+y_2+y_3=1.
\end{align}
The objective function $\tilde{f}(x)$ in \eqref{eq:ExampleFun1} can be modelled by the convex combination method proposed in \cite{Hanif} as:
\be
\label{eq:ObjCCEx1}
\tilde{f}(x)=7.5\lambda_0+2.5\lambda_1+10\lambda_2+5\lambda_3+7.5\lambda_4+5\lambda_5
\ee 
Now, since the function defined in \eqref{eq:ExampleFun1} above is right-continuous but not lower semi-continuous, no minima exist. Hence, let us solve the following maximization problem:
\be
\label{eq:Ex1}
\max_{\mathbf{x}} \hspace{0.5 cm}\sum_{n=1}^{N_{\text{var}}}{\tilde{f}(x_n)},
\ee 
by the two models, where $\mathbf{x}=[x_1,...,x_{N_{\text{var}}}]^T$ . Obviously, the function $\tilde{f}(x)$ attains a maximum value of 10 at $x=1$ because there are no constraints in problem \eqref{eq:Ex1}. However, if we maximize the sum of the separable functions over $N_{\text{var}}$ variables, we can assess the performance of the two models. The problems were solved by using CPLEX solver and the results are given in Table~\ref{tab:NumResultsI} with the time taken to reach the solution for different values of $N_{\text{var}}$. OOM in the table stands for \textit{out of memory}.
\begin{table*}[t]
\centering
\caption{Comparison of the numerical solutions of problem \eqref{eq:Ex1}}
\label{tab:NumResultsI}
\begin{tabular}{lcccc}\hline
&\multicolumn{2}{c}{Incremental method}&\multicolumn{2}{c}{Convex Combination method}\\
&\multicolumn{2}{c}{\eqref{eq:xIMEx1}, \eqref{con:Deltay_1Ex1}, \eqref{eq:ObjIMEx1}}&\multicolumn{2}{c}{\eqref{eq:xCCEx1}, \eqref{con:LambdaEx1}, \eqref{eq:ObjCCEx1}}\\\hline
$N_{\text{var}}~\times 10^3$&Objective ($\times 10^3$)&Time (sec)&Objective ($\times 10^3$)&Time (sec)\\\hline
1&10&0.03&10&0.44\\
5&50&0.08&50&1.26\\
10&100&0.14&100&2.32\\
20&200&0.30&200&4.84\\
50&500&0.73&500&12.37\\
100&1,000&1.45&OOM&\\
250&2,500&4.51&OOM&\\\hline
\end{tabular}
\end{table*}
One can note that the incremental method with the proposed modification in \eqref{eq:ObjDeltaContRight} could reach the solutions in much less time than that taken by using the convex combination method. One main reason for this is the less number of variables required by the incremental method. Actually, by using the incremental method one needs $3\times N_{\text{var}}$ continuous variables and $2\times N_{\text{var}}$ binary variables, while by using the convex combination method one needs $6\times N_{\text{var}}$ continuous variables and $3\times N_{\text{var}}$ binary variables. Of course, this does not mean that the incremental method with the modification in \eqref{eq:ObjDeltaContRight} for discontinuous functions that are right-continuous is always better than other methods. Actually, this depends on the nature of the problem. All what we can say is that the incremental method can be modified to suit discontinuous functions that are right-continuous. Further, the incremental method for modelling the PWL functions require fewer variables than those needed for the convex combination method, especially for discontinuous functions, and this could have considerable influence in reducing the computational time of such problems.  
\subsection{Left-Continuous}
Let us consider now a PWL function that is only left-continuous at all breaking points but not from the right, as: 
\be
\label{eq:PWLfConLeft}
\tilde{g}(x)=\left\{\begin{array}{lr}
m_1x+d_1,&a_0\leq x\leq a_1\\
m_2x+d_2,&a_1< x\leq a_2\\
\vdots&\\
m_Kx+d_K,&a_{K-1}< x\leq a_K
\end{array}\right\}.
\ee 
Obviously, we cannot use the jump increment as before because when $y_k=a_k-a_{k-1}$ and thus $\beta_k=1$, there is no jump. What we can do here is to reverse the incremental variable. We commence by modifying the representation of the variable $x$ in \eqref{eq:xIM} into:
\be
\label{eq:xIMLeft}
x=a_K-\sum_{k=1}^K{\tilde{y}_k}.
\ee  
The only discrepancy between $\tilde{y}_k$ in the above and $y_k$ in \eqref{eq:xIM} is that $\tilde{y}_k$ starts from the end of the domain of $x$. Thus, instead of incrementing $a_0$ by a sequence of $y_k$ to reach $x$, we decrement $a_K$ by $\tilde{y}_k$ to reach $x$. Then, the constraints in \eqref{con:Deltay_1} and \eqref{con:Delta} are changed to:
\begin{align}
\label{con:DeltaLeftCont}
\tilde{y}_1\leq& a_K-a_{K-1},\hspace{0.75cm}\tilde{y}_K\geq0\nonumber\\
\tilde{y}_k\geq&(a_{K-k+1}-a_{K-k})\tilde{\beta}_k,\hspace{1cm} \forall k\in\{1,..K-1\}\nonumber\\
\tilde{y}_k\leq&(a_{K-k+1}-a_{K-k})\tilde{\beta}_{k-1},\hspace{1cm}\forall k\in\{2,..K\}. 
\end{align}
Just like the model in \eqref{con:Deltay_1} and \eqref{con:Delta}, $\tilde{\beta}_k=1$ only when $\tilde{y}_k=a_{K-k+1}-a_{K-k}$. For example, if $\tilde{\beta}_1=1$ and hence $\tilde{y}_1=a_K-a_{K-1}$, then $x=a_K-\tilde{y}_1=a_{K-1}$, and so on. Now, the PWL function can be described by:
\be
\label{eq:ObjDeltaContLeft}
\tilde{g}(x)=\tilde{g}(a_K)+\sum_{k=1}^K{\left(-m_{K-k+1}\tilde{y}_k+\tilde{\Delta}_k\tilde{\beta}_k\right)},
\ee
where $\tilde{\Delta}_k$ denotes the jump at the point $a_{K-k}$, and it is given by $\tilde{g}(a_{K-k})-(m_{K-k+1}a_{K-k}+d_{K-k+1})$. 
\begin{remark}
The model described by \eqref{eq:xIMLeft}, \eqref{con:DeltaLeftCont}, and \eqref{eq:ObjDeltaContLeft} for discontinuous PWL functions that are left-continuous is still locally ideal.
Obviously, the polytope of the LP relaxation of \eqref{con:DeltaLeftCont} has all extreme points integral in $\tilde{\beta}_k~\forall~k$, as proved in \cite{Padberg}.
\end{remark}
\subsubsection{Numerical Example}
Let us consider the PWL function $\tilde{g}(x)$ defined as follows:
\be
\label{eq:ExampleFun2}
\tilde{g}(x)=\left\{\begin{array}{lr}
-5x+7.5,&0\leq x\leq 1\\
-5x+15,&1< x\leq 2\\
-2.5x+12.5,&2< x\leq 3
\end{array}\right\}.
\ee 
In order to build the incremental model of this PWL function by using \eqref{eq:xIMLeft}, \eqref{con:DeltaLeftCont}, we need three continuous variables $\tilde{y}_k~\in~\mathbb{R}, ~\forall~k\in\{1,2,3\}$, and two binary variables $\tilde{\beta}_k~\in~\{0,1\}, ~\forall~k\in\{1,2\}$. Then, the model can be described by:
\be
\label{eq:xIMEx2}
x=3-\tilde{y}_1-\tilde{y}_2-\tilde{y}_3
\ee
with the constraints:
\begin{align}
\label{con:Deltay_1Ex2}
\tilde{\beta}_1\leq \tilde{y}_1\leq 1\nonumber\\
\tilde{\beta}_2\leq \tilde{y}_2\leq \tilde{\beta}_1\nonumber\\
0\leq \tilde{y}_3\leq \tilde{\beta}_2.
\end{align}
The objective function $\tilde{g}(x)$ in \eqref{eq:ExampleFun2} can be expressed by the proposed model in \eqref{eq:ObjDeltaContLeft} as:
\be
\label{eq:ObjIMEx2}
\tilde{g}(x)=5+2.5\tilde{y}_1+5\tilde{y}_2+5\tilde{y}_3-2.5\tilde{\beta}_1-7.5\tilde{\beta}_2.
\ee
As was done before, we compare the proposed incremental model with the convex combination model described in \cite{Hanif} which was proved to be locally ideal. Actually, by using the model described in \cite{Hanif}, the PWL function in \eqref{eq:ExampleFun2} will have the same model used before for the PWL function in \eqref{eq:ExampleFun1}, namely the model described by \eqref{eq:xCCEx1} and \eqref{con:LambdaEx1}. The reason behind that is the similarity between the functions $\tilde{f}(x)$ and $\tilde{g}(x)$. Obviously, the two functions are identical except at the breaking points. Thus, the objective function $\tilde{g}(x)$ in \eqref{eq:ExampleFun2} can be modelled by using the convex combination method proposed in \cite{Hanif} as:
\be
\label{eq:ObjCCEx2}
\tilde{g}(x)=7.5\lambda_0+2.5\lambda_1+10\lambda_2+5\lambda_3+7.5\lambda_4+5\lambda_5
\ee 
Now, since the function defined in \eqref{eq:ExampleFun2} above is left-continuous, let us solve the following minimization problem:
\be
\label{eq:Ex2}
\min_{\mathbf{x}} \hspace{0.5 cm}\sum_{n=1}^{N_{\text{var}}}{\tilde{g}(x_n)},
\ee 
by the two models, where $\mathbf{x}$ is as before. Obviously, the function $\tilde{g}(x)$ attains a minimum value of 2.5 at $x=1$ because there are no constraints in problem \eqref{eq:Ex2}. One more time, we minimize the sum of the separable functions over $N_{\text{var}}$ variables in order to assess the performance of the two models. The problems were solved by using CPLEX solver and the results are given in Table~\ref{tab:NumResultsII} with the time taken to reach the solution for different values of $N_{\text{var}}$. As before, OOM stands for \textit{out of memory}
\begin{table*}[t]
\centering
\caption{Comparison of the numerical solutions of problem \eqref{eq:Ex2}} 
\label{tab:NumResultsII}
\begin{tabular}{lcccc}\hline
&\multicolumn{2}{c}{Incremental method}&\multicolumn{2}{c}{Convex Combination method}\\
&\multicolumn{2}{c}{\eqref{eq:xIMEx2}, \eqref{con:Deltay_1Ex2}, \eqref{eq:ObjIMEx2}}&\multicolumn{2}{c}{\eqref{eq:xCCEx1}, \eqref{con:LambdaEx1}, \eqref{eq:ObjCCEx2}}\\\hline
$N_{\text{var}}~\times 10^3$&Objective ($\times 10^3$)&Time (sec)&Objective ($\times 10^3$)&Time (sec)\\\hline
1&2.5&0.02&2.5&0.55\\
5&12.5&0.08&12.5&1.64\\
10&25&0.16&25&3.29\\
20&50&0.31&50&10.20\\
50&125&0.76&125&21.98\\
100&250&1.53&OOM&\\
250&625&3.99&OOM&\\\hline
\end{tabular}
\end{table*}
Again, we can note that the incremental method with the proposed modification in \eqref{eq:ObjDeltaContLeft} could reach the solutions in much less time than that taken by using the convex combination method. 
\section{Optimization over many variables with Binary Indicators}
The optimization problem in \eqref{eq:BinInd} appears in many applications such as unit commitment and gas production \cite{Sridhar}. Although many authors tried to suggest a PWL representation of the non-separable objective function, such as the authors in \cite{Vielma1}, we restrict ourselves to the case of the separable functions given by:
\begin{align}
\label{eq:SepBinInd}
\min_{\mathbf{x}}& \hspace{0.5 cm}\sum_{n=1}^N{f_n(x_n)}\nonumber\\
\text{s.t.}&\hspace{0.5 cm}\mathbf{x}\in\Omega \cup \{\mathbf{0}\},
\end{align}
where $f_n(x_n)$ is a PWL function. Let us assume that the feasibility region $\Omega$ is given by:
\be
\Omega=\{\mathbf{x}\in\mathbb{R}^N:~x_n\in[\underline{x}_n,\overline{x}_n],~\forall n\in\{1,\dots,N\}\}.
\ee
Let us also assume that each interval $[\underline{x}_n,\overline{x}_n]$ is partitioned into $K$ segments with the same number of breaking points $\underline{x}_n=a_{n_0},\dots,a_{n_K}=\overline{x}_n$ for, as hinted in \cite{Farias3}, if they are not the same one can always choose extra breaking points on segments of the same slope. In order to represent the non-convex feasibility region of each $x_n$, one would use a binary indicator $\alpha_n\in\{0,1\},~\forall n\in\{1,\dots,N\}$, and add the following constraint to the optimization problem and the PWL model:
\be
\label{eq:BinIndBound}
a_{n_0}\alpha_n\leq x_n\leq a_{n_K}\alpha_n .
\ee
For convenience, let us drop the index $n$ because we are interested in the feasibility region of one variable. Actually, the authors in \cite{Sridhar} showed that any PWL model with the constraint in \eqref{eq:BinIndBound} will not be locally ideal, in general. Instead, they proposed to modify the PWL incremental model described by \eqref{eq:xIM}, \eqref{con:Deltay_1}, \eqref{con:Delta} and \eqref{eq:ObjDelta}, by modifying the constraint in \eqref{con:Deltay_1} as \cite{Sridhar}:
\be
\label{con:Deltay_1Sridhar}
y_1\leq(a_1-a_0)\alpha.
\ee 
Then, they proved that the model with the modification above is locally ideal for the incremental model. The authors in \cite{Sridhar}, assumed that $a_0=0$ and $f(a_0)=0$. This looks intuitive because if this is not the case, the function $\tilde{f}(x)$ defined on $[a_0,a_K]$ can always be shifted horizontally or vertically. However, this assumption may not make sense. The bounds in \eqref{eq:BinIndBound} are usually used to indicate that if the variable $x$ does not lie in that interval, then the cost function will be zero. Thus, the idea behind such optimization problems is to optimize some cost function in $x$ over its domain or to set $x=0$, and thus $\tilde{f}(x)=0$. Before proceeding with  the discussion, let us define the LP relaxation of the model proposed in \cite{Sridhar} for PWL functions with the binary indicators for the incremental model as:
\begin{align}
\mathbf{F}_{IM}=\{(&x,\mathbf{y},\boldsymbol{\beta},\alpha)\in\mathbb{R}\times\mathbb{R}^K\times[0,1]^{K-1}\times[0,1]:\nonumber\\
&\eqref{eq:xIM}, \eqref{con:Delta}, \eqref{con:Deltay_1Sridhar}\text{ are satisfied}\}.
\end{align}
Now, we are ready to present our remarks on this model.
\newline
As mentioned before, the authors in \cite{Sridhar} proved that $\mathbf{F}_{IM}$ is locally ideal. However, the model in $\mathbf{F}_{IM}$ has an \textit{incorrect} extreme point when $a_0=0$. Consider the point $p^\ast=(x,\mathbf{y},\boldsymbol{\beta},\alpha)=(0,\mathbf{0},\mathbf{0},1)$. Obviously, $p\in\mathbf{F}_{IM}$ and is extreme. Thus, the solution to the LP relaxation may contain this point which is incorrect. What we mean here by \textit{incorrect} extreme point can be explained as follows. The point $p^\ast$ has a zero value of the variable $x$ and unity value in the variable $\alpha$. Since $\alpha$ is a binary variable that is used to indicate whether $x$ takes real values in its domain or not, the variable $\alpha$ should be zero when $x=0$.
\newline
What we propose to solve this is to modify \eqref{eq:xIM} and \eqref{eq:ObjDelta} to:
\be
\label{eq:xIMmod}
x=a_0\alpha+\sum_{k=1}^K{y_k},
\ee  
and
\be
\label{eq:ObjDeltaMod}
\tilde{f}(x)=f(a_0)\alpha+\sum_{k=1}^K{m_ky_k},
\ee
respectively, in addition to the modification in \eqref{con:Deltay_1Sridhar}. Let the LP relaxation of the new proposed model be:
\begin{align}
\mathbf{P}_{IM}=\{(&x,\mathbf{y},\boldsymbol{\beta},\alpha)\in\mathbb{R}\times\mathbb{R}^K\times[0,1]^{K-1}\times[0,1]:\nonumber\\
&\eqref{con:Delta}, \eqref{con:Deltay_1Sridhar},\eqref{eq:xIMmod},\text{ are satisfied}\}.
\end{align}
Note that the formulation $\mathbf{P}_{IM}$ is still locally ideal, since the polytope of the LP relaxation of the region described by the inequality constraints has all extreme points integral in $\boldsymbol{\beta}$ and $\alpha$ as proved in \cite{Sridhar}. However, in the case $a_0=0$ and $f(a_0)=0$, the model will boil down to the one in $\mathbf{F}_{IM}$, which has the incorrect extreme point shown before. 
\newline
Another alternative is just to use the lower bound $a_0\alpha\leq x$ in the incremental method. In order to comment on this model, let us define the following LP relaxation:
\begin{align}
\mathbf{P}^\prime_{IM}=\{&(x,\mathbf{y},\boldsymbol{\beta},\alpha)\in\mathbb{R}\times\mathbb{R}^K\times[0,1]^{K-1}\times[0,1]:\nonumber\\
&\eqref{eq:xIM},\eqref{con:Delta}, \eqref{con:Deltay_1Sridhar},~x\geq a_0\alpha\text{ are satisfied}\}.
\end{align}
Actually, the model in $\mathbf{P}^\prime_{IM}$ is locally ideally but it has also an incorrect extreme point, as shown in the following result.
\begin{proposition}
$\mathbf{P}_{IM}\subset \mathbf{P}^\prime_{IM}$.
\end{proposition}
\begin{proof}
It is clear that any point $p\in \mathbf{P}_{IM}$ is also in $\mathbf{P}^\prime_{IM}$, but the opposite is not necessarily true. Consider the point $p_0=(x,\mathbf{y},\boldsymbol{\beta},\alpha)=(a_0,\mathbf{0},\mathbf{0},0)$. Obviously, $p_0\in\mathbf{P}^\prime_{IM}$ but $p_0\notin\mathbf{P}_{IM}$. 
\end{proof}
This shows that $\mathbf{P}_{IM}$ is tighter than $\mathbf{P}^\prime_{IM}$. Besides, the model described by $\mathbf{P}^\prime_{IM}$ has at least one incorrect extreme point, that is the point given in the proof above.
\begin{remark}
The lower bound $a_0\alpha\leq x$ is redundant when the equality constraint in \eqref{eq:xIMmod} is used due to the non-negativity of the variables $y_k$.
\end{remark}
In a nutshell, when the PWL function is defined over $x\in[a_0,a_K]$, the bounding constraints in \eqref{eq:BinIndBound} are not required. Actually, the upper boundary $x\leq a_K\alpha$ destroys the ideality property of the model, as proved in \cite{Sridhar}. For the incremental model we propose to replace the lower bound with the equality constraints in \eqref{eq:xIMmod} and \eqref{eq:ObjDeltaMod} when $a_0\neq0$ or $f(a_0)\neq0$. On the other hand, for the case when $a_0=f(a_0)=0$ the formulation by incremental model contains an incorrect extreme point. Fortunately, this case is unlikely to happen in applications, unless some shifting is performed. Thus, we recommend not to shift the PWL function when it is defined over some interval that does not contain the zero.   
\section{CONCLUSION}
The incremental method to model the PWL function is proved to give locally ideal formulations just like the SOS2 method. In this work, we modified the incremental method to suit special class of discontinuous functions that are either right-continuous at all breaking points or left-continuous. The rationale behind this modification is the fact that the incremental method requires fewer continuous and binary variables. We also showed by numerical results that the incremental method could reduce the computational time significantly, due to the reduction in the number of variables. The proposed modification does not destroy the desired local ideality property of the incremental method. 
\newline 
Then, we discussed the optimization problem over several univariate PWL functions defined on intervals, which do not contain the zero, with binary indicators. The common technique to solve such optimization problems is to include a boundary constraint like the one in \eqref{eq:BinIndBound}. It was proved in \cite{Sridhar} that the upper bound constraint may destroy the local ideality of the formulation in any model. In this work, we proposed a simple modification to tighten the formulation formed by the incremental method. 
\newline
Of course, this does not mean that the incremental method always gives the best formulation of the PWL functions in optimization problems because this depends on the nature of every particular problem. We merely want to say that the incremental method with the proposed modifications requires fewer variables, especially discontinuous functions. Thus, it could be more efficient than the convex combination or SOS2 methods.   

\end{document}